\documentclass[12pt, letterpaper]{article}

\usepackage[margin=1.25in]{geometry}
\usepackage[utf8]{inputenc}

\usepackage{color}
\definecolor{darkgreen}{rgb}{0,0.5,0}
\usepackage[
        colorlinks, citecolor=darkgreen,
        backref,
        pdfauthor={S. Gajovi\'c},
        pdftitle={Experiments on curves with sharp Coleman bound}
]{hyperref}

\usepackage{tikz,enumerate,caption,stmaryrd,amsfonts,amssymb,systeme,comment}
\usetikzlibrary{matrix,positioning,arrows,shapes,chains,calc,automata}
\usetikzlibrary{decorations.pathreplacing}
\usepackage{amssymb}
\usepackage{amsmath}
\usepackage{mathtools}
\usepackage{amsthm}
\usepackage{tikz-cd}
\usepackage[OT2,T1]{fontenc}
\usepackage{url}
\DeclareSymbolFont{cyrletters}{OT2}{wncyr}{m}{n}
\DeclareMathSymbol{\Sha}{\mathalpha}{cyrletters}{"58}

\usepackage{cleveref}
\crefformat{section}{§#2#1#3}

\newcommand{\Sel}{\operatorname{Sel}}
\newcommand{\tors}{\operatorname{tors}}

\newcommand{\Res}{\text{Res}}
\newcommand{\rad}{\text{rad}}

\newtheorem{theorem}{Theorem} 
\newtheorem{proposition}[theorem]{Proposition}
\newtheorem{corollary}[theorem]{Corollary}
\newtheorem{lemma}[theorem]{Lemma}

\title{Curves with sharp Chabauty-Coleman bound}
\author{Stevan Gajovi\'c  }
\date{}

\begin{document}

\maketitle

\begin{abstract}

We construct curves of each genus $g\geq 2$ for which Coleman's effective Chabauty bound is sharp and Coleman's theorem can be applied to determine rational points if the rank condition is satisfied. We give numerous examples of genus two and rank one curves for which Coleman's bound is sharp. Based on one of those curves, we construct an example of a curve of genus five whose rational points are determined using the descent method together with Coleman's theorem. 
\end{abstract}

\tableofcontents


\section{Introduction} \label{introduction}

One of the fundamental problems in the arithmetic of curves and abelian varieties is to find ways to use $p$-adic information on curves to determine rational points. It is known that $p$-adic points have a richer structure than rational points, and, in principle, are easier to determine. Thus, it is natural to try to use the knowledge of $p$-adic points in studying rational points. There are several such approaches, e.g., the Hasse principle and its generalizations, or bounding ranks of elliptic curves by Selmer groups.\\ 

The method of Chabauty and Coleman, nowadays also called "abelian Chabauty," is a powerful $p$-adic method, and it is often very successful in determining rational points on curves whose Jaconians satisfy a certain rank condition. Coleman's "effective Chabauty" theorem (Theorem \ref{C-Bound}) gives an upper bound on the number of rational points on these curves. However, it is rare that the direct application of Coleman's theorem succeeds in determining all rational points because the bound from Coleman's theorem is not sharp in most cases.\\

In fact, there are only a few curves known for which Coleman's bound is sharp. Hence it is of interest to study the question highlighted by Coleman's result: Regardless of the rank of its Jacobian, when does the
number of rational points on $C$ meet or exceed Coleman's bound at $p$? We will call curves for which the rank condition is satisfied, and for which Coleman's bound is sharp for some prime number $p$ of good reduction, \emph{sharp curves at $p$}. The curves such that the number of their known rational points meets Coleman's bound regardless of the rank condition will be called \emph{potentially sharp curves at $p$}. Note that we cannot use the method of Chabauty and Coleman to determine the set $C(\mathbb{Q})$ for potentially sharp curves $C$ because we did not check the rank condition. However, if the rank condition holds for a potentially sharp curve, then the curve is sharp; hence, we determined its rational points. Curves that exceed the bound will be called \emph{excessive at $p$}. Curves for which there is a prime $p$ such that the curve is sharp at $p$, potentially sharp at $p$, or excessive at $p$ are called sharp, potentially sharp, and excessive, respectively. We note that the rank of the Jacobians of excessive curves is at least $g$.\\

We concisely introduce the method in \cref{C-C-method}, and the computations of the rank of Jacobians of hyperelliptic curves in \cref{Descent-Jacobians}. In \cref{Descent-Rational-Points}, we briefly present the descent method for determining rational points on curves. The descent method is usually used to descend to curves of genus zero and one. Thus, in \cref{Chabauty+descent}, we construct a curve whose rational points are determined by using descent to curves of genus two, where one of them is sharp.\\

After mentioning two known sharp curves in \cref{Known curves}, we give various examples of sharp curves of genus two in \cref{Finite family}, and one with the smallest number of rational points in \cref{Small curve}. In \cref{Lower bounds rank} we list two curves that violate Coleman's bound, and we use this information to determine their ranks.\\

In the third section, we start by presenting sharp curves of genus three, four, and five in \cref{Small large genera}. Using results on the existence of primes in short intervals, we construct infinitely many potentially sharp curves of each genus $g\geq 2$ in \cref{Potential examples}. We give examples of sharp curves of genus $g=4$ and $g=5$ based on this construction. \\

All computations were done in Magma \cite{Magma}. 

\subsection{The method of Chabauty and Coleman and its history} \label{C-C-method}

In 1922, in \cite{Mordell}, Louis J. Mordell proved that the group of rational points of an elliptic curve $E/\mathbb{Q}$ is finitely generated. At the end of the article, Mordell observes that curves of genus $g\geq 2$ have only finitely many rational points. The statement that for a nice (smooth, projective, and geometrically irreducible) curve $C$ of genus $g\geq 2$ and a number field $K$ it holds that $\#C(K)$ is finite was named the "Mordell conjecture". After Mordell's theorem it was a natural question whether the group $A(K)$ is finitely generated for all abelian varieties $A/K$, where $K$ is any number field. In 1928, André Weil proved in \cite{Weil} that for all abelian varieties $A$ over the number field $K$ we have 
$$A(K)\cong \mathbb{Z}^r\oplus A(K)_{\tors},$$
where $A(K)_{\tors}$ is the torsion subgroup of $A(K)$ and $r\in\mathbb{N}_0$ is called the rank of $A(K)$. 

In the 1940s, Claude Chabauty proved the Mordell conjecture for curves that satisfy a certain condition.

\begin{theorem}(Chabauty 1941, \cite{Chabauty}) Let $C$ be a nice curve of genus $g\geq 2$. Let $J(C)$ be its Jacobian, and let $r$ be the rank of $J(C)$ over $\mathbb{Q}$. If $r<g$, then $C(\mathbb{Q})$ is finite.
\end{theorem}

We briefly discuss the idea behind the proof. We may assume that $C(\mathbb{Q})\neq\emptyset$ and embed $C(\mathbb{Q})$ via an Abel-Jacobi map into $J(C)(\mathbb{Q})$, and do the same over the field $\mathbb{Q}_p$, where $p$ is a prime of good reduction for $C$. That means that $\overline{C}$, the curve that is obtained by reduction of $C$ modulo $p$, remains a smooth curve over $\mathbb{F}_p$. We have the following commutative diagram.

\begin{tikzcd}
C(\mathbb{Q}) \arrow[r, hook] \arrow[d, hook] & J(C)(\mathbb{Q}) \arrow[d, hook] \\
C(\mathbb{Q}_p) \arrow[r, hook] & J(C)(\mathbb{Q}_p)
\end{tikzcd}

The set $J(C)(\mathbb{Q}_p)$ has the structure of a $g$-dimensional $p$-adic manifold, containing $C(\mathbb{Q}_p)$ as a 1-dimensional submanifold. The discrete group $J(C)(\mathbb{Q})$ also sits inside $J(C)(\mathbb{Q}_p)$, and we take its closure in the $p$-adic topology of $J(C)(\mathbb{Q}_p)$, denoted by $\overline{J(C)(\mathbb{Q})}$, to make it a submanifold of $J(C)(\mathbb{Q}_p)$. Then the dimension of $\overline{J(C)(\mathbb{Q})}$ is not greater than $r$, so heuristically for dimensional reasons, it should have finite intersection with $C(\mathbb{Q}_p)$, which was what Chabauty proved. This intersection contains the set $C(\mathbb{Q})$ (more precisely, $C(\mathbb{Q})$ injects into it), implying the finiteness of $C(\mathbb{Q})$.

This theorem remained the most significant result towards the Mordell conjecture until Gerd Faltings proved the conjecture unconditionally, i.e. for all nice curves of genus $g\geq 2$, in \cite{Faltings}. It seemed that Chabauty's theorem lost its value. But, that is not true because Robert Coleman in 1985 found a way to use Chabauty's approach to state and prove an effective version of the theorem.

\begin{theorem}(Coleman 1985, \cite{ColemanEC}) \label{C-Bound} Let $C$ be a nice curve of genus $g\geq 2$. Let $J(C)$ be its Jacobian, and let $r$ be the rank of $J(C)(\mathbb{Q})$. Let $p$ be a prime of good reduction for $C$. If $r<g$, and $p>2g$, then
$$\#C(\mathbb{Q})\leq \#\overline{C}(\mathbb{F}_p)+2g-2.$$
\end{theorem}

Coleman invented a theory of $p$-adic integration on curves in \cite{ColemanI}, which has properties that we would expect from integration, such as linearity in integrands and additivity in endpoints. When two endpoints are in the same residue disc, i.e., when both have the same reduction modulo $p$, we can compute the integral in the expected way by expressing the integrand as a power series in a local parameter and then integrating term by term. There are more important properties, such as Frobenius equivariance, and unexpectedly, path-independence, unlike the real and complex case. The last two properties turn out to be very convenient, and all properties together make Coleman integrals practically computable, which makes Chabauty's method effective. Using Coleman integration, we can construct a locally analytic function $\rho:C(\mathbb{Q}_p)\longrightarrow\mathbb{Q}_p,$ which vanishes on $C(\mathbb{Q})$. We can estimate the number of zeros in each residue disc, and summing all estimates gives us the bound in Theorem \ref{C-Bound}. 

We can furthermore investigate the function $\rho$ in each residue disc (and sometimes combine results with some other methods) to provably determine $C(\mathbb{Q})$. This locally analytic function $\rho$ is the main ingredient in determining rational points on curves, and it is a main object of research nowadays to extend the method to the complementary case $r\geq g$. This is an explicit version of a non-abelian extension of this method, a program initiated by Kim in \cite{Kim2005MFG} and \cite{Kim2009Selmer} when $r\geq g$. Quadratic Chabauty is possible when $r=g$, because then, instead of linear functionals (integrals), we use quadratic functions ($p$-adic heights) to determine rational points. The first explicit examples can be found in \cite{BBM1}.  

By a careful investigation inside each residue disc, Michael Stoll in \cite{StollImprovedBound} proved a theorem that improves Coleman's bound.

\begin{theorem}(Stoll 2006, \cite{StollImprovedBound}) \label{S-rank} Let $C$ be a nice curve of genus $g\geq 2$. Let $J(C)$ be its Jacobian, and let $r$ be the rank of $J(C)$ over $\mathbb{Q}$. Let $p$ be a prime of good reduction for $C$. If $r<g-1$, and $p>2r+2$, then
$$\#C(\mathbb{Q})\leq \#\overline{C}(\mathbb{F}_p)+2r.$$
\end{theorem}

As a corollary, we have the following rank statement for sharp curves. We know that $r<g$. If $r<g-1$, then Coleman's bound cannot be sharp.

\begin{corollary} \label{Rank-corollary}
Let $C$ be a sharp curve of genus $g\geq 2$. Then the rank of the Jacobian of $C$ over $\mathbb{Q}$ is $r=g-1$.
\end{corollary}

If a curve satisfying the rank condition is not sharp, this means that looking at the local information at $p$ is not sufficient to determine its rational points. Nevertheless, it is often possible to compute the rational points by combining the Chabauty information at $p$ with $v$-adic information for other primes $v$, for instance, obtained through the Mordell-Weil sieve. The Mordell-Weil sieve was introduced in \cite{Scharaschkin}, and well explored in the paper \cite{Bruin_Stoll_2010}. We also note that the Mordell-Weil sieve can be successfully combined with the Quadratic Chabauty method, see, e.g., \cite{BBM2}. 

One reference for more details about the method of Chabauty and Coleman is \cite{WMC-BP}. The notes \cite{SiksekOhrid} contain some very explicit examples, including the combination of the method with the Mordell-Weil sieve. 

\subsection{On the computation of the rank of the Jacobian of hyperelliptic curves} \label{Descent-Jacobians}

The problem of computing the rank of abelian varieties is extremely difficult in general and still open. It is related to one of the most important conjectures in arithmetic geometry. This is the Birch and Swinnerton-Dyer conjecture, which states that the rank of an abelian variety can be computed analytically, namely, as the order of vanishing of its $L$-function at $s=1$ (for more details see, e.g., \cite{TateBSD}). The latter quantity is called the analytic rank. Thus, if one is willing to assume the Birch and Swinnerton-Dyer conjecture, then the rank of the Jacobian of a hyperelliptic curve can be computed as an analytic rank. We can compute the $L$-function of hyperelliptic curves in Magma using the algorithm by Dokchitser \cite{LFunctionsMagma}, and we can evaluate it at $s=1$ as well as its derivatives. We need to be careful that we can compute these values only up to some precision. In general, we cannot prove that some value is zero, although we can verify that it is very close to zero.

Algebraically, there has been a lot of progress since Weil's theorem, and there are methods to compute the rank that work in many cases, but none of them is guaranteed to work, even for the simplest case of elliptic curves.
One of the first explicit computations for dimension greater than one appeared in the paper \cite{Gordon-Grant}, followed by \cite{Schaefer}. It was later generalized and implemented for hyperelliptic curves in Magma by Stoll, \cite{Stoll-Implementing-2-descent}.

Let $J$ be the Jacobian of a hyperelliptic curve $C:y^2=f(x)$. From the Kummer exact sequence
$0\longrightarrow J[2]\longrightarrow J  \overset{[2]}{\longrightarrow}  J\longrightarrow 0,$ and the long exact sequence of Galois cohomology, we get the following commutative diagram ($M_{\mathbb{Q}}$ is the set of places of $\mathbb{Q}$):

\begin{tikzpicture}
    \node[left] at (0,0) (1) {0};
    \draw [->] (1) -- (1,0);
    \node[right] at (1,0) (2) {$J(\mathbb{Q})/2J(\mathbb{Q})$};
    \draw [->]  (2) -- (5.5,0);
    \node[right] at (5.5,0) (3)  {$H^1(\mathbb{Q}, J[2])$};
    \draw [->]  (3) -- (9.5,0);
    \node[right] at (9.5,0) (4)  {$H^1(\mathbb{Q}, J)[2]$};
    \draw [->]  (4) -- (13,0);
    \node[right] at (13,0) (5)  {0};
    \node[left] at (0,-1.75) (6) {0};
    \draw [->] (6) -- (0.5,-1.75);
    \node[right] at (0.5,-2) (7) {$\displaystyle\prod_{v\in M_{\mathbb{Q}}} J(\mathbb{Q}_v)/2J(\mathbb{Q}_v)$};
    \draw [->]  (4.25,-1.75) -- (5.1,-1.75);
    \node[right] at (5,-2) (8)  {$\displaystyle\prod_{v\in M_{\mathbb{Q}}} H^1(\mathbb{Q}_v, J[2])$};
    \draw [->]  (8.4,-1.75) -- (9.15,-1.75);
    \node[right] at (9,-2) (9)  {$\displaystyle\prod_{v\in M_{\mathbb{Q}}} H^1(\mathbb{Q}_v, J)[2]$};
    \draw [->]  (12.35,-1.75) -- (13,-1.75);
    \node[right] at (13,-1.75) (10)  {0};
    \draw [->] (2) --(2.25,-1.25);
    \draw [->] (3) --(6.7,-1.25);
    \draw [->] (4) --(10.7,-1.25);
    \draw [->] (3) -- (9.5,-1.25) node[midway,sloped,above] {$\alpha$};
\end{tikzpicture}


The kernel of the map $\alpha$ is \emph{the Selmer group} $\Sel^{(2)}(J/\mathbb{Q})$, and the kernel 
$$\Sha(\mathbb{Q},J):=\ker\left( H^1(\mathbb{Q}, J)\longrightarrow \prod_{v\in M_{\mathbb{Q}}} H^1(\mathbb{Q}_v, J)\right)$$ is called \emph{the Shafarevich-Tate group}.

We have the following exact sequence from the diagram above
$$0\longrightarrow J(\mathbb{Q})/2J(\mathbb{Q}) \longrightarrow \Sel^{(2)}(J/\mathbb{Q}) \longrightarrow \Sha(\mathbb{Q},J)[2]\longrightarrow 0.$$

It is a well-known fact that $\Sel^{(2)}(J/\mathbb{Q})$ is a finite group. The idea of the proof, which works for more general Selmer groups and can be found in, e.g., \cite[C.4]{Hindry-Silverman}, is to construct an injection into a finite commutative group indexed by bad and infinite primes. In \cite[C.4]{Hindry-Silverman}, we also see that the Selmer group is, in principle, effectively computable. We have 
$$\dim_{\mathbb{F}_2}J(\mathbb{Q})/2J(\mathbb{Q})=r+\dim_{\mathbb{F}_2}J(\mathbb{Q})[2]\leq \dim_{\mathbb{F}_2}\Sel^{(2)}(J/\mathbb{Q}).$$
The difference between the two sides is exactly $\dim_{\mathbb{F}_2}\Sha(\mathbb{Q},J)[2]$.

We can easily determine $J(\mathbb{Q})[2]$; it depends on the factorization of $f$ over $\mathbb{Q}$. Thus, we can obtain an upper bound on the rank of $J(\mathbb{Q})$.

For the lower bound, we search for rational points on $J$, up to some bounded height, and hope that eventually, we will find the same number of independent points as the upper bound for the rank. If the curve has many rational points, then we might be able to find a lower bound on the rank, see the remark in \cref{Lower bounds rank}. The idea to use Coleman's bound to obtain the lower bound for the rank with the infinite families of curves as the examples and the improvements on the lower bound can be found in \cite{OntheCCbound}.

From the discussion above, we easily see why the algorithm is not guaranteed to work, but it still works very well in practice. We also see that the (partial) knowledge of $\Sha(\mathbb{Q},J)$ can improve upper bounds, e.g., as in \cite[Chapter 8]{Stoll-Implementing-2-descent}. 

\subsection{Introduction to the descent method} \label{Descent-Rational-Points}

The descent method is one of the first methods ever used for determining rational points on curves. It is still very applicable, especially for hyperelliptic curves. The basics of this method are explained in more detail in \cite{StollOhrid}. For explicit descent on hyperelliptic curves, see \cite{Bruin-Stoll}. The descent method is based on the property of unique factorization in $\mathbb{Z}$. In a sense, it is a generalization of the fact that if a product of two numbers is a square, then these two numbers have the same squarefree part. We can bound the possible greatest common divisor of these numbers, which gives only finitely many possibilities for their squarefree parts. In principle, we have a strategy to reduce all possible factorizations to a finite number.

Let us make it more precise. Suppose that $C/\mathbb{Q}$ is a hyperelliptic curve given by the equation
$$C:y^2=f_1(x)f_2(x),$$
where $f_1$ and $f_2$ are non-constant coprime integral polynomials, with at least one of $\deg(f_1)$ and $\deg(f_2)$ even. Then, for every point $(x_0,y_0)\in C(\mathbb{Q})$, we have $f_1(x_0)\neq 0$ or $f_2(x_0)\neq 0$, and there is a unique squarefree $d\in\mathbb{Z}$, and $z_0,t_0\in\mathbb{Q}$, such that 
$$f_1(x_0)=dz_0^2, \quad f_2(x_0)=dt_0^2.$$
Denote by $C_d$ the curve given by equations 
$$C_d:f_1(x)=dz^2, \quad f_2(x)=dt^2;$$
then there is a covering $\pi_d:C_d\longrightarrow C$ given by 
$$(x,z,t)\mapsto(x,dzt).$$
It follows that
$$C(\mathbb{Q})=\bigcup_{d\in S}\pi_d(C_d),$$
where $S\subset \mathbb{Z}$ denotes the set of all squarefree numbers. This set is infinite, but we can reduce it to a finite set. Let us assume that $f_1$ and $f_2$ are monic polynomials (otherwise, we have to include primes that divide their leading coefficients). As these polynomials are coprime, their resultant is a non-zero integer $R$. The resultant gives information on possible common roots. So, for all primes $p$ such that $p\nmid R$, the resultant of $f_1$ and $f_2$ over $\mathbb{F}_p$ is not zero, implying that $f_1$ and $f_2$ do not have a common root in $\mathbb{F}_p$. It means that $f_1$ and $f_2$ are "coprime at $p$", i.e., that it is not possible that $p$ divides the squarefree part $d$. So, the number of possible primes that can divide the squarefree part $d$ is finite, and we need to investigate a finite number of curves $C_d$ (which, in principle, should be easier than doing so for $C$) to find all rational points on $C$.

\subsection{Acknowledgements}

The author is funded by the DFG-Grant MU 4110/1-1. Part of this research was done during
the "Summer school in computational number theory" in Bristol and during a visit to Boston
University where the author was partially supported by the Diamant PhD Travel Grant. The
author thanks Steffen Müller for his support, checking examples and numerous comments and
corrections on the text, Jennifer Balakrishnan, Francesca Bianchi, C\'eline Maistret, Michael Stoll, and Jaap Top for useful discussions, Pieter Moree for significant help on obtaining the special case of Bertrand's postulate used here, Lazar Radi\v{c}evi\'{c} for many suggestions and checking examples, Alex Best for helpful suggestions, Oana Adascalitei for help and support when creating examples, and Boston University for their hospitality and providing Magma on their computer. The author thanks Noam Elkies and Kirti Joshi for bringing attention to examples related to the work and ideas mentioned in this paper.
 The author thanks the anonymous referees for carefully reading the article and many useful comments.
\section{Examples of sharp curves of genus two}

\subsection{Two known sharp curves in genus two} \label{Known curves}

We first mention two known examples.\\

\noindent {\bf Example 1.}(\cite{Grant}) This is the first known example of a sharp curve, so we see that it took almost ten years after Coleman's paper \cite{ColemanEC} until the first example appeared. Let $C$ be the curve given by the equation
$$C: y^2=x(x-1)(x-2)(x-5)(x-6).$$
This curve has good reduction at $p=7$, and it has eight $\mathbb{F}_7$-rational points
$$\overline{C}(\mathbb{F}_7)=\{(0,0),(1,0),(2,0),(3,\pm1),(5,0),(6,0),\infty\}.$$
In \cite{Gordon-Grant}, it is computed that $J(C)(\mathbb{Q})\cong \mathbb{Z}\times\left(\mathbb{Z}/2\mathbb{Z}\right)^4$. We can apply Theorem \ref{C-Bound} to obtain that $C(\mathbb{Q})\leq 8+2=10$. We can find ten rational points, so
$$C(\mathbb{Q})=\{(0,0),(1,0),(2,0),(3,\pm6),(5,0),(6,0),(10,\pm120),\infty\}.$$

\noindent {\bf Example 2.}(\cite{Triangles}) This example is miraculous because it arises
from a geometric problem. Namely, this curve occurs in analysing the problem whether there exists a pair of one right triangle and one isosceles triangle, both with rational sides, having the same area and perimeter. More details about the problem can be found in \cite{Triangles}. The curve $C$ is given by the equation
$$C:y^2=(x^3-x+6)^2-32.$$
We compute that the rank of its Jacobian over $\mathbb{Q}$ is equal to $r=1$, that $C$ has good reduction at $p=5$, and that its reduction modulo 5 has exactly eight points. All conditions for Theorem \ref{C-Bound} are satisfied; thus we conclude
$$C(\mathbb{Q})=\left\{(0,\pm2),(1,\pm2),(-1,\pm2),\left(\dfrac{5}{6},\pm\dfrac{217}{216}\right),\infty_{\pm}\right\}.$$

\subsection{A finite family of sharp curves of genus two} \label{Finite family}

We construct new examples of sharp curves of genus two. Let $C$ be a curve. We focus on genus two curves because the computation of the rank becomes more difficult for larger genus. Thus, the bound is given by 
$$\#C(\mathbb{Q})\leq \#\overline{C}(\mathbb{F}_p)+2.$$
Since we want this bound to be sharp, we would like that two residue discs have extra points (this condition is natural due to the hyperelliptic involution), and other ones have exactly one rational point. The easiest way to achieve this is to require a small number of residue discs. We will construct a suitable monic polynomial $f(x)$ of degree 5. One natural choice for a prime $p$ to consider is $p=11$, because $x^5\equiv 0,\pm1 \pmod{11}$. If we want to include other monomials in $x$, we want them to have a coefficient divisible by 11, to control the polynomial $f(x)$ modulo 11 easily. It remains to pick the constant term. We want some constant for which we can find rational points, but not too many modulo 11. Quadratic residues modulo 11 are in the set $\{0,1,3,4,5,9\}$, so if we pick the constant term to be 9 modulo 11, we have that 
$$f(x)\equiv x^5+9\equiv \{8,9,10\}\pmod{11}.$$
This implies that 
$$\overline{C}(\mathbb{F}_{11})=\{(0,\pm3),\infty\}.$$
When the rank of the Jacobian of $C$ over $\mathbb{Q}$ is less than two, by Coleman's bound, $C$ can have at most 5 rational points. \\

Note, as in \cite{Grant}, that we should check if a curve $C$ has absolutely irreducible Jacobian. If this were not the case, then $J(C)$ would be isomorphic to the product of two elliptic curves over some number field $K$, and one of them would have rank zero, providing an easier way to determine rational points on $C$. In \cite{Grant}, David Grant together with Jaap Top, proved absolute simplicity of the Jacobian of the curve from Example 1 by proving that its $L$-function cannot be a product of two $L$-functions of elliptic curves. In the meantime, in \cite{Howe-Zhu}, Everett W. Howe and Hui J. Zhu found a criterion when an abelian variety over a finite field is absolutely simple.

\begin{lemma} (\cite[Lemma 8]{Howe-Zhu})\label{HZ-irred}
Let $q$ be a prime power and $n>2$ an integer. Suppose that $\pi$ is an ordinary Weil $q$-number (the minimal polynomial for $\pi$ is a characteristic polynomial of Frobenius of an ordinary abelian variety). Let $K=\mathbb{Q}(\pi)$, $K^+$ its maximal real subfield, and $n=[K^+:\mathbb{Q}]$. Suppose that \\
(1) the minimal polynomial of $\pi$ is not of the form $x^{2n}+ax^n+q^n$,\\
(2) the field $K^+$ has no proper subfields other than $\mathbb{Q}$,\\
(3) the field $K^+$ is not the maximal real subfield of a cyclotomic field.\\
Then the isogeny class corresponding to $\pi$ consists of absolutely simple abelian varieties.\\
\end{lemma}
We finally give a number of examples.\\

\begin{proposition} \label{example3} Let $C_k:y^2=f(x)$ be a hyperelliptic curve over $\mathbb{Q}$, where 
$$f(x)=x^5+11x^4+(11k+3)^2,\hspace*{2mm} k\in\{0,1,2,3,7,10,11,12,15,21,22,31,40,42,44,47,50\}$$ 
or
$$f(x)=x^5+11x^4+(11k-3)^2,\hspace*{2mm} k\in\{1,4,9,15,16,17,19,27,28,31,40,41,42,43\}.$$
(1) Then $\#C_k(\mathbb{Q})=5$, and in the first case
$$ C_k(\mathbb{Q})= \{(0,\pm(11k+3)),(-11,\pm(11k+3)),\infty\},$$
whereas in the second case
$$ C_k(\mathbb{Q})= \{(0,\pm(11k-3)),(-11,\pm(11k-3)),\infty\}.$$
(2) All curves $C_k$ from (1) have absolutely simple Jacobian.
\end{proposition}

\begin{proof}
(1) In both cases, we run the Magma command {\tt RankBound} to compute an upper bound on the rank of the Jacobian of curves $C_k$ for $0\leq k \leq 50$. We listed above all $k$'s in both cases for which the upper bound is less than two. As we saw, $p=11$ is a prime of good reduction of all curves $C_k$ and $\#\overline{C_k}(\mathbb{F}_{11})=3$, so we conclude by Theorem \ref{C-Bound} that $\#C_k(\mathbb{Q})\leq 5$. The fact that we have five obvious rational points on each $C_k$ finishes the proof. All listed curves $C_k$ are sharp and have rank $r=1$.\\\\
(2) Using Magma and Lemma \ref{HZ-irred} we check that all $J(C_k)$ are absolutely simple by finding a prime $p$ such that the reduction of $J(C_k)$ modulo $p$ is absolutely simple.\\
\end{proof}

\subsection{An example with descent and a sharp curve} \label{Chabauty+descent}

There are many examples of applying descent from curves of higher genus to curves of genus zero or one in the literature. Lack of examples of curves whose rational points can be determined by applying descent to curves of genus at least two motivated us to construct one such curve.

\begin{proposition}
Let $C/\mathbb{Q}$ be the hyperelliptic curve given by the equation
$$C:y^2=f(x):=(x^6+11x^5+64x+729)(x^5+11x^4+64).$$
The set of rational points on $C$ is
$$C(\mathbb{Q})=\{(0,\pm216), (-11,\pm40), \infty\}.$$
\end{proposition}

\begin{proof}
As we can see, the polynomial $f$ is monic, so when applying the descent method we need to compute the resultant of two polynomials, whose product is $f$, which is
$$R:=\Res(x^6+11x^5+64x+729,x^5+11x^4+64))=3^{30}.$$
We are only interested in the radical of the resultant, since it is enough to consider squarefree numbers, and $\rad(R)=3$.
Thus, after applying the descent, we know that $x$ corresponds to an $x$-coordinate of a rational point on one of curves
$$x^5+11x^4+64=dz^2, \quad x^6+11x^5+64x+729=dt^2$$
for some $d\in\{-3,-1,1,3\}.$ \\

We first prove that $d<0$ is impossible. If $x^5+11x^4+64<0$, then $x<0$ and so 
$$x(x^5+11x^4+64)+729>0,$$ 
giving a contradiction. 

If $d=3$, then we can use the linear change of coordinates $X=3x$, $Y=3^3z$, and then the equation of the curve $$C_3: x^5+11x^4+64=3z^2$$ transforms to $$C'_3:X^5+33X^4+64\cdot 3^5=Y^2.$$ For the curve $C'_3$, we compute in Magma that $J(C'_3)(\mathbb{Q})=\left<0\right>$. We know that $C'_3(\mathbb{Q})$ embeds into $J(C'_3)(\mathbb{Q})$, so $C'_3$ has at most one rational point: this is the one at infinity.\\

If $d=1$, then we consider the hyperelliptic curve $$C_1: x^5+11x^4+64=z^2,$$ 
which is one of the curves in Proposition \ref{example3}, and we already know 
$$C_1(\mathbb{Q})=\{(0,\pm8),(-11,\pm8),\infty\}.$$
We see that all of these points give rise to the rational points of $C$, and that
$$C(\mathbb{Q})=\{(0,\pm216),(-11,\pm40),\infty\}.$$
\end{proof} 

\noindent{\bf Comment.} Let $r$ be the rank of the Jacobian of $C$ over $\mathbb{Q}$. Assuming the Generalized Riemann Hypothesis, Magma gives that $2\leq r\leq 4$. So, in principle, we would be able to apply the abelian Chabauty method to $C$. However, this seems to be more difficult than the approach in the proof, which is furthermore unconditional. As we can check in Magma, the Jacobian of $C$ is absolutely simple.

\subsection{A sharp curve with the smallest possible number of points} \label{Small curve}

We give an example of a sharp curve that satisfies the property that it has the smallest number of rational points amongst all sharp curves of genus $g\geq 2$. For a sharp curve $C$ of genus $g\geq 2$ it holds $\#C(\mathbb{Q})\geq 3$. Indeed, we know that in that case
$$\#C(\mathbb{Q})=\#\overline{C}(\mathbb{F}_{p})+2g-2\geq \#\overline{C}(\mathbb{F}_{p})+2\geq 2.$$ Therefore, the set $C(\mathbb{Q})$ is non-empty, implying that also $\overline{C}(\mathbb{F}_{p})$ is non-empty, giving the desired conclusion.
If we want to construct such a curve, we want a curve of genus two with only one residue disc. If we consider a hyperelliptic curve $C:y^2=f(x)$, where $f$ is a monic polynomial of degree 5, we already know that we have one residue disc at infinity. So, we want to construct a curve having no other residue discs, meaning that the reduction over $\mathbb{F}_p$ can have only the point at infinity. We will use a similar strategy as before, so that 
$$\overline{C}:y^2=x^5+c/\mathbb{F}_{11}.$$
If $c\equiv7\pmod{11}$, then $\overline{C}$ has only one rational point, $\overline{C}(\mathbb{F}_{11})=\{\infty\}$ because 6, 7, and 8 are not quadratic residues modulo 11, and $C$ has good reduction at $p=11$. There should be two more rational points on $C$. Let one of them be $P=(a,b)$. Then $P$ maps to infinity modulo 11. Furthermore, we know that the denominator of $a$ is divisible by $11^2$ since otherwise $f(a)$ is not a square of a rational number. Let us try to find an example such that $a$ is already a square of a rational number and that $a$ is a zero of a polynomial $f(x)-x^5$, with $f$ satisfying the conditions above.\\

Let $C$ be the hyperelliptic curve over $\mathbb{Q}$ defined by
$$C: y^2=x^5+121x-4.$$
Magma gives us that $r$, the rank of the Jacobian of $C$ over $\mathbb{Q}$, is zero or one. We can apply Theorem \ref{C-Bound}, and by the previous discussion,
$$C(\mathbb{Q})=\left\{\left(\dfrac{4}{121},\pm \dfrac{2^5}{11^5}\right),\infty\right\}.$$

We also know that $r=1$ in this case by Corollary \ref{Rank-corollary}. We could also conclude this in another way: Using Magma we see that the torsion subgroup is trivial, which is impossible when $r=0$ and $\#C(\mathbb{Q})>1$. 

\subsection{Examples on improving lower rank bounds in Magma} \label{Lower bounds rank}

We now present two examples of excessive curves, for which we can improve Magma's lower bounds for the rank. In both examples, we can determine the rank of the curve, although Magma gives inconclusive information. 

Again, it is desirable to have a small number of residue discs, and in particular, we would like only one residue disc. This time, we will consider $p=5$ and the curve
$$C:y^2=8x^6-314x^5+3250x^4-10000x^3+64x=x(2(x-25)(4x-25)(x^3-8x^2)+64).$$ 
The reduction of $C$ modulo 5 is
$$\overline{C}:y^2=x(3x^5+x^4-1)/\mathbb{F}_5$$
We easily see that $C$ has good reduction at $p=5$ and that $\overline{C}(\mathbb{F}_5)=\{(0,0)\}$. We can find at least five points on this curve:
$$\left\{(0,0),\left(\dfrac{25}{4},\pm20\right),(25,\pm40)\right\}\subset C(\mathbb{Q}),$$
so $C$ is excessive. Magma gives us that the rank $r$ of the Jacobian of $C$ satisfies $0\leq r\leq 2$. Hence, we conclude that $r=2$.\\

Note that we do not find the set $C(\mathbb{Q})$ here because Chabauty's condition $r<g$ is violated. It might be possible to compute $C(\mathbb{Q})$ using Quadratic Chabauty.\\

The following example is conditional, i.e., for the rank computations in Magma we need to assume the Generalized Riemann Hypothesis. We can do the same for the following curve
$$C: y^2=x^5-12(121x-1)(121x-4)$$
looking at the good prime $p=11$. This curve has at least five rational points, so it is excessive. Magma gives us that the rank of its Jacobian over $\mathbb{Q}$ is between zero and two, thus, it is two.\\

\noindent{\bf Remark.} We see that a curve $C/\mathbb{Q}$ with many rational points (at least $2g-2$) might be potentially sharp or excessive. There is an algorithm to check that. We need to determine all prime numbers $p$ for which $C$ can be potentially sharp or excessive at $p$. We use the Hasse-Weil bounds and the number of known rational points to give an upper bound for primes $p$, so we need to check finitely many of them. Therefore, if we find that the curve is potentially sharp or excessive, we obtain a lower bound for the rank, $r\geq g-1$, or $r\geq g$, respectively. \\

Note that here we only considered primes of good reduction. There are results similar to Coleman's bound when we consider primes of bad reduction, and these can be found in \cite{Lorenzini-Tucker} and improved, with a bound in Stoll's style, in \cite{Katz-DZB}. It would be interesting to investigate the case of bad reduction in the future. In this case, we might be able to construct infinite families of examples with sharp bounds. We could then allow multiple Weierstrass points in the same residue discs, which makes computations of ranks of curves easier and hope that we can determine the ranks of Jacobians for an infinite number of curves.

\section{Examples of sharp curves of higher genus}

We present a few sharp curves of genus $g>2$. There is a way to construct potentially sharp curves, which works for any genus $g$ and the construction is contained in \cref{Potential examples}, including examples of sharp curves of genus $g=4$ and $g=5$, one of each. In some cases, there is an easier way to find examples of sharp curves, and it is presented in the following subsection.

\subsection{Concrete examples of sharp curves of genus three, four and five} \label{Small large genera}

Let $g$ be a positive integer. In the case where $2g+1$ or $2g+3$ is a prime number, we can construct examples of potentially sharp curves in the following ways.\\

If $2g+1=p$, where $p$ is some prime number, then $x^{2g+1}-x\equiv 0\pmod{p}$ for all $x\in\mathbb{F}_p$. If $c$ is any quadratic nonresidue modulo $p$, then for the curve 
$$\overline{C}:y^2=x^{2g+1}-x+c/\mathbb{F}_p$$
we have $\overline{C}(\mathbb{F}_p)=\{\infty\}$.
So, all possible rational points on the curve $C$ reduce modulo $p$ to $\infty$, thus are not integral and have denominators divisible by $p$. Let $a_1,\ldots,a_{g-1}$ be integers not divisible by $p$, with distinct absolute values. Let $b\in\mathbb{Z}$ be any inverse of $(a_1\dots a_{g-1})^2$ modulo $p$. The following curve is potentially sharp:
$$C:y^2=x^{2g+1}+b(a_1^2-p^2x)\cdots(a_{g-1}^2-p^2x)(c-x).$$

If $2g+3=p>3$ is prime, then we can use the property that there are two consecutive quadratic nonresidues modulo $p$. If $p\equiv 1\pmod{4}$, then between 2 and $p-2$ there are $\frac{p-1}{2}$ quadratic nonresidues implying that two of them must be consecutive. If $p\equiv 3\pmod{4}$, $-1$ and $-4$ are quadratic nonresidues, so if $-2$ or $-3$ is a quadratic nonresidue, we have the conclusion. If not, then 2 and 3 are consecutive quadratic nonresidues. Denote by $c$ and $c+1$ two consecutive quadratic nonresidues modulo $p$. Then the curve
$$\overline{C}:y^2=x^{2g+2}+c/\mathbb{F}_p$$
has two $\mathbb{F}_p$-points, both at infinity because, for $x\in\mathbb{F}_p$, we have that
$$x^{2g+2}+c\equiv \{c,c+1\}\pmod{p}.$$
We construct examples from this curve. Let $a_1,\ldots,a_{g-1}$ be distinct integers not divisible by $p$.  Let $b\in\mathbb{Z}$ be any inverse of $a_1\cdots a_{g-1}$ modulo $p$. The following curve is potentially sharp:
$$C':y^2=x^{2g+2}+b(a_1-px)\cdots(a_{g-1}-px)c.$$

We summarize the conclusions above into the theorem.

\begin{theorem}
Let $g\geq 2$ be a positive integer such that $2g+1$ or $2g+3$ is a prime number. Define $C$ and $C'$ as above in each case. If $r$, the rank of the Jacobian of $C$ over $\mathbb{Q}$, satisfies $r<g$, then $r=g-1$, and
$$C(\mathbb{Q})=\left\{\left(\dfrac{a_1^2}{p^2},\pm\dfrac{a_1^{2g+1}}{p^{2g+1}}\right),\ldots,\left(\dfrac{a_{g-1}^2}{p^2},\pm\dfrac{a_{g-1}^{2g+1}}{p^{2g+1}}\right),\infty\right\}.$$
If $r'$, the rank of the Jacobian of $C'$ over $\mathbb{Q}$, satisfies $r'<g$, then $r'=g-1$, and
$$C'(\mathbb{Q})=\left\{\left(\dfrac{a_1}{p},\pm\dfrac{a_1^{g+1}}{p^{g+1}}\right),\ldots,\left(\dfrac{a_{g-1}}{p},\pm\dfrac{a_{g-1}^{g+1}}{p^{g+1}}\right),\infty_{\pm}\right\}.$$
\end{theorem}

\begin{proof}
Follows from the above and Corollary \ref{Rank-corollary}.
\end{proof}

We give one example of a sharp curve in each genus $g\in\{3,4,5\}$. For these curves we are able to check the rank condition $r<g$. However, we need to assume the Generalized Riemann Hypothesis for the rank computations in Magma. The curves are as follows:
$$C_3:y^2= x^7 - (49x-1)(49x-36)(x+1), \hspace*{2mm} r(J(C_3)(\mathbb{Q}))=2,$$
$$ C_3(\mathbb{Q})=\left\{\left(\dfrac{1}{49},\pm\dfrac{1}{7^7}\right), \left(\dfrac{36}{49},\pm\dfrac{6^7}{7^7}\right), \infty\right\};$$
$$C_4:y^2 = x^{10} - (11x-3)(11x-4)(11x-6), \hspace*{2mm} r(J(C_4)(\mathbb{Q}))=3,$$
$$C_4(\mathbb{Q})=\left\{\left(\dfrac{3}{11},\pm\dfrac{3^5}{11^5}\right),\left(\dfrac{4}{11},\pm\dfrac{4^5}{11^5}\right), \left(\dfrac{6}{11},\pm\dfrac{6^5}{11^5}\right), \infty_{\pm}\right\};$$
$$C_5:y^2 = x^{12} - (13x-1)(13x-2)(13x-3)(13x-12), \hspace*{2mm} r(J(C_5)(\mathbb{Q}))=4,$$
$$C_5(\mathbb{Q})=\left\{\left(\dfrac{1}{13},\pm\dfrac{1}{13^6}\right), \left(\dfrac{2}{13},\pm\dfrac{2^6}{13^6}\right), \left(\dfrac{3}{13},\pm\dfrac{3^6}{13^6}\right), \left(\dfrac{12}{13},\pm\dfrac{12^6}{13^6}\right), \infty_{\pm}\right\}.$$

\subsection{Bertrand's Postulate for primes modulo 8} \label{Bertrand mod 8}

For our construction, we will need a stronger version of Bertrand's Postulate. We formulate one version first, and then we cover smaller cases.

\begin{proposition} 
Let $n\geq 15$ be a positive integer. In the interval $[n,2n)$ there is a prime number $p$ such that $p\equiv 5\pmod{8}$.
\end{proposition}
\begin{proof}
Let $\mathbb{P}$ denote the set of prime numbers in $\mathbb{Z}$. For any coprime integers $k$ and $l$, we define the function 
$$\theta(x;k,l)=\sum_{p\in\mathbb{P},\hspace*{1.5mm} p\leq x,\hspace*{1.5mm} p\equiv l\hspace*{-3mm}\pmod{k}} \log(p).$$
Theorem 1, for $k\leq 13$, from \cite{Ramare-Rumely}, states that
$$\max_{1\leq y\leq x}\left|\theta(y;k,l)-\dfrac{y}{\varphi(k)}\right|\leq \varepsilon\dfrac{x}{\varphi(k)},$$
where one can take $\varepsilon=0.00456$ for $x\geq 10^{10}$. It implies that when $y=x$
$$\left|\theta(x;8,5)-\dfrac{x}{4}\right|\leq 0.00456 \dfrac{x}{4}\leq 0.005 \dfrac{x}{4}\implies 0.995 \dfrac{x}{4}\leq \theta(x;8,5)\leq 1.005\dfrac{x}{4}$$
for $x\geq 10^{10}$.
If $x\geq 10^{10}$, then also
$$\theta(x;8,5)\leq 1.005\dfrac{x}{4} < 0.995 \dfrac{2x}{4}\leq \theta(2x;8,5)$$
giving that $\theta(2x;8,5)>\theta(x;8,5)$ for $x\geq 10^{10}$. It follows that there is a prime congruent to 5 modulo 8 between $x$ and $2x$. We need to prove the statement for the remaining interval $[15,10^{10}]$. We can easily find a list of primes congruent to 5 modulo 8 such that the next one is larger than half of the previous one (e.g., we can use Magma for it), which completes the proof. For instance, we can take the list to be
$$10000000061, 5000000141, 2500000117, 1250000077, 625000069, 312500077, 156250093, $$
$$78125141, 39062581, 19531381,
9765757, 4882957, 2441573, 1220797, 610429, 305237, $$ 
$$152629, 76333, 38189, 19141, 9613, 4813, 2437, 1229, 653, 349, 181, 101, 53, 29.$$
\end{proof}

\begin{proposition} \label{B-primes}
Let $n\geq 2$ be a positive integer. Then in the interval $[n,2n)$ there is a prime number $p$ such that $p\equiv 3\pmod{8}$ or $p\equiv 5\pmod{8}$.
\end{proposition}
\begin{proof}
By the previous proposition, we only need to check small cases, but the statement is trivially true since 3, 5, 11, 19 are primes that fill in the missing gaps.
\end{proof}

As a corollary, in each interval $[n,2n)$, where $n>1$ is a positive integer, there is a prime $p$ such that 2 is not a quadratic residue modulo $p$, which is the property we will use in our construction of curves in the following subsection.

\subsection{Potentially sharp curves of genus $g\geq 2$} \label{Potential examples}

Let $g\geq 2$ be a positive integer. We will construct potentially sharp hyperelliptic curves $C:y^2=f(x)$ of genus $g$, where $f$ is a monic polynomial. For simplicity, we present the construction and ideas in a simpler way. We add a comment on how we can slightly generalize it using the same ideas after Theorem \ref{Main examples} and use these generalized ideas to construct concrete examples.\\

By our version of Bertrand's postulate, Proposition \ref{B-primes}, there is a prime number $p$, congruent to 3 or 5 modulo 8, that is contained in the interval $(2g+2,4g+4)$. Define the polynomial 
$$Q(x)=x^{2g+2}-x^{\frac{p-1}{2}}+x^{2g+2-\frac{p-1}{2}}+1.$$
Then we have the following congruence modulo $p$:
$$Q(x)\equiv \begin{cases}
2x^{2g+2} \pmod{p}, \text{if } \left(\dfrac{x}{p}\right)=1,\\
2 \pmod{p}, \text{if } \left(\dfrac{x}{p}\right)=-1,\\
1 \pmod{p}, \text{if } p\mid x.
\end{cases}
$$
Since $p$ is chosen so that 2 is quadratic nonresidue modulo $p$, it follows that
$$\overline{C}(\mathbb{F}_p)=\{(0,\pm1),\infty_{\pm}\},$$
if $\overline{C}$ is the hyperelliptic curve over $\mathbb{F}_p$, defined by $y^2=Q(x)$.\\

For each $s\leq g$, $s\in\mathbb{N}$, we will construct curves $C_s:y^2=H_s$, whose reduction modulo $p$ is $\overline{C}$ and which will have at least $4+2s$ rational points. Four rational points on $C_s$ will be from the set $\{(0,\pm1),\infty_{\pm}\}$, and we need to add $2s$ points, i.e. $s$ values of $x$ for which $H_s(x)$ is a square of a rational number. We will construct $H_s$ such that new rational points on $C_s$ are integral, so that all $x$-coordinates will be integers divisible by $p$. We note that then $C_{g-1}$ is a potentially sharp curve and $C_g$ excessive. \\

Let $a_1,\ldots,a_{s}\in\mathbb{Z}$ be distinct and non-zero. We will construct the polynomial $H_s$ such that $H_s(pa_1)$, $H_s(pa_2)$, $\ldots$, $H_s(pa_s)$ are squares of integers. The strategy is as follows. All monomials $x^k$ with $k>s$ will be replaced by 
$$t_s(x^k):=x^{k-s}(x-pa_1)(x-pa_2)\cdots(x-pa_s).$$
We extend this transformation linearly so that we can compute $t_s(P)\in\mathbb{Z}[x]$ for all polynomials $P\in\mathbb{Z}[x]$ that have only monomials in $x$ of degree greater than $s$. This transformation has two important properties, namely for all such $P\in\mathbb{Z}[x]$, we have
$$t_s(P)\equiv P \pmod{p},$$
$$t_s(P)(0)=t_s(P)(pa_1)=\cdots=t_s(P)(pa_s)=0.$$
We can apply this transformation to the higher degree part of $Q(x)$, e.g. to the polynomial $x^{2g+2}-x^{\frac{p-1}{2}}$ (note that $\frac{p-1}{2}>g\geq s$). We need to deal with the remaining part 
$1+x^l$, where $l:=2g+2-\frac{p-1}{2}$. If $l>s$, then we can apply $t_s$ to $x^l$ and define a polynomial 
$$Z_s(x):=t_s(x^{2g+2}-x^{\frac{p-1}{2}}+x^l)+1.$$
If $s\geq l$, the polynomial $1+x^l$ will be transformed into a part of a square of some polynomial. In this way, we can assure that $H_s(pa_i)$ is a square for $i=1,\ldots,s$. Denote $m\in\mathbb{N}$ the number for which $ml\leq s<(m+1)l$. We use the following lemma.
\begin{lemma} \label{Coeffs}
There exist integers $c_1$, $c_2$, $\ldots$, $c_m$ such that 
$$(1+c_1x^l+c_2x^{2l}+\cdots+c_mx^{ml})^2\equiv 1+x^l \pmod{p,x^{s}},$$
meaning that these two polynomials agree modulo $p$ up to degree $s$.
\end{lemma}
\begin{proof}
We prove this lemma constructively. We can start by taking $c_1=\frac{p+1}{2}$ (or any other number congruent modulo $p$ to it). Then, looking at coefficients with $x^{jl}$, where $j\in\{2,\ldots,m\}$ we have 
$$2c_j+S_j\equiv 0\pmod{p},$$
where $S_j$ is some polynomial in the previous coefficients $c_1$, $\ldots$, $c_{j-1}$. Since $p$ is odd, 2 is invertible and we can determine $c_j$ modulo $p$.
\end{proof}

If $c_j$ are the coefficients from the previous lemma, then 
$$(1+c_1x^l+c_2x^{2l}+\cdots+c_mx^{ml})^2=G_s(x)+L_s(x),$$
where $\deg(L_s)\leq s$ and $L_s(x)\equiv 1+x^l\pmod{p}$, and all monomials in $x$ in $G_s$ have degree at least $s+1$. Note $\deg(G_s)=2ml\leq 2s<2g+2$. Define the polynomial $Z_s$ as
$$Z_s(x)=t_s(x^{2g+2}-x^{\frac{p-1}{2}}-G_s(x))+G_s(x)+L_s(x).$$
In both cases ($l>s$ and $s\geq l$) for all $i=1,\ldots,s$, we have 
$$Z_s(pa_i)=G_s(pa_i)+L_s(pa_i)=(1+c_1p^la_i^l+c_2p^{2l}a_i^{2l}+\ldots+c_mp^{ml}a_i^{ml})^2=:b_i^2,$$
where, if $l>s$, we have $c_1=\cdots=c_m=0$, i.e. $b_1=\cdots=b_s=1$, $Z_s(0)=1$, and
$$Z_s(x)\equiv x^{2g+2}-x^{\frac{p-1}{2}}+x^l+1=Q(x)\pmod{p}.$$

We can change $Z_s$ by any polynomial $px(x-pa_1)\cdots(x-pa_s)R(x)$, where $R\in\mathbb{Z}[x]$ has degree at most $2g-s$, and the new curve will still have all known rational points and the same reduction modulo $p$. So, for any choice as above, define
$$C_s:y^2=H_s(x),\hspace*{2mm} H_s(x):=Z_s(x)+px(x-pa_1)\cdots(x-pa_s)R(x).$$
Note that we can take $b_i\equiv 1 \pmod{p}$, so $b_i\neq 0$, and we indeed have at least $4+2s$ rational points on $C_s$:
$$\{\infty_{\pm}, (0,\pm1), (pa_1,\pm b_1), \ldots, (pa_s, \pm b_s)\} \subset C_s(\mathbb{Q}).$$

Therefore, we constructed many examples of potentially sharp curves $C_{g-1}$ and curves $C_g$ of large rank. Let $r_s$ be the rank of the Mordell-Weil group of the Jacobian of $C_s$ over $\mathbb{Q}$.  We summarize this in the following theorem.

\begin{theorem}\label{Main examples}\hfill\\
(1) Let $s=g-1$, i.e., let $C_{g-1}$ be a curve defined as above. If $r_{g-1}<g$, then $r_{g-1}=g-1$, and
$$C_{g-1}(\mathbb{Q})=\{\infty_{\pm}, (0,\pm1), (pa_1,\pm b_1), \ldots, (pa_{g-1}, \pm b_{g-1})\}.$$
(2) Let $s=g$, i.e., let $C_g$ be a curve defined as above. Then $r_g\geq g$.
\end{theorem}

\begin{proof}
(1) By construction, $C_{g-1}$ is potentially sharp. If the rank condition is satisfied, then $C_{g-1}$ is sharp, hence has exactly $2g+2$ points and rank $r_{g-1}=g-1$.\\
(2) The curve $C_g$ is constructed to be excessive, so $r_g\geq g$.
\end{proof}

\noindent{\bf Comment.} In the spirit of the ideas in the construction, we can change the curves in the following ways.\\
(1) The transformation $t_s$ can be slightly changed to ($k>e_1+\dots+e_s$)
$$t_{e_1,\dots,e_s}(x^k)=x^{k-e_1-\dots-e_s}(x-pa_1)^{e_1}(x-pa_2)^{e_2}\cdots(x-pa_s)^{e_s}.$$
(2) The numbers $a_1,\dots, a_s$ do not have to be integers; we can allow them to be rational numbers such that $a_i=\frac{b_i}{c_i}$, GCD$(b_i,c_i)=1$, $p\mid b_i$ for all $i=1,\dots,s$. But we need to ensure that the curve still has two points at infinity. Note that the resulting polynomial, after a linear change of variables, might not be monic anymore.\\
(3) We can change anything modulo $p$, so that the reduction of the curve stays the same, as long as we do not change the number of known rational points. For example, the constant term of the polynomial can be any square of an integer congruent to 1 modulo $p$, or we can change the polynomials constructed in Lemma \ref{Coeffs}, or increase their degree until it is not greater than $g$, and similar constructions.\\

We now want to construct some concrete examples of sharp curves from the previous family. To this end we need to check is whether the rank assumption is satisfied. In a sense, we want to find curves with many rational points and small rank of its Jacobian. \\

\noindent{\bf Remark.} There is a program started by Noam Elkies, and continued by Michael Stoll and Andreas K\"{u}hn dealing with a very similar question: Can we construct curves of genus 2 with many rational points whose differences are contained in the same arithmetic progression inside the Jacobian? Such curves have small rank despite having many rational points. In particular, we can apply the method of Chabauty and Coleman for these curves, at least with respect to the subgroup of the Jacobian, which contains the arithmetic progression. In many cases, this subgroup is a finite index subgroup of the Jacobian. Some examples can be found in \cite{Elkies} and \cite{Stoll-presentation}.

We now look at some of their examples and we determine whether they are sharp. In \cite{Elkies}, we can, e.g., find the family 
$$C(b): y^2 = (x^3 - x^2 + b x + (b+1))^2 - 4 b (b+1) x.$$
The curve $C(b)$ has rational points $\infty_{\pm}$, $(0,\pm(b+1))$, $(1,\pm1)$, and $(-1,\pm(2b+1))$. We denote $q:=[\infty_+ -\infty_-]$. Then $[(0,-(b+1))-(0,b+1)]=5q$, $[(1,1)-(1,-1)]=13q$, $[(-1,2b+1)-(-1,-(2b+1)]=29q$. One special curve, for $b=-\frac{5}{2}$, has six pairs of points in the arithmetic progression. The curve $C(-\frac{5}{2})$ has to more pairs of rational points, $(3,\pm6)$, and $(\frac{1}{2},\pm\frac{7}{8})$, and $[(3,6)-(3,-6)]=61q$, $[(\frac{1}{2},\frac{7}{8})-(\frac{1}{2},-\frac{7}{8})]=83q$. 

The rank of $J(C(-\frac{5}{2}))(\mathbb{Q})$ is equal to 1, and it is almost sharp, since it can be proven using the method of Chabauty and Coleman combined with the Mordell-Weil sieve that $\#C(\mathbb{Q})=12$, whereas $\#\overline{C}(\mathbb{F}_7)=11$. 

There is another curve, found by Stoll, $C:y^2 = (x+2) (x^2-3x+6) (4x^3+8x^2-3x+3)$, that is sharp and is interesting because it has 13 rational points and still rank 1. In this case, the differences of non-Weierstrass points and their images under the hyperelliptic involution do not form an arithmetic progression, but an arithmetic progression modulo the torsion subgroup of the Jacobian. 

In the end, we mention the curve from \cite{Stoll-presentation}, $C:y^2=25x^6 + 20x^5 - 76x^4 - 134x^3 + 124x^2 + 96x + 9$. This curve is interesting because its Jacobian contains a rational point of canonical height $\approx 0.000298$, which is the smallest known positive canonical height on a Jacobian surface over $\mathbb{Q}$. It turns out that this curve is sharp.\\

Based on the construction of curves $C_{g-1}$ from Theorem \ref{Main examples}, we searched for examples of sharp curves using Magma. We were unable to find sharp examples of genus two or three. For these genera, the generic rank of the subgroup generated by differences of the rational points on the curve appears to be at least $g$. We give one example of a sharp curve of genus four and one example of a sharp curve of genus five. For both examples, we assume the Generalized Riemann Hypothesis to compute the rank. We did not check the rank conditions of curves of genus at least six, since this is an extremely computationally difficult task.\\

We start with a curve of genus four, and we choose the prime $p=11$. Then the reduction modulo 11 is the curve $\overline{C}:y^2=x^{10}+1$. Since there are no small powers of $x$, we only need to add more rational points using some convenient $t$ transformation. We consider the curve
$$C:y^2=x^4(x-11)^2(x-22)^2(x-33)^2+1.$$
This curve has the property that the element $[\infty_--\infty_+]$ of the Jacobian has finite order 5, increasing the chances to obtain smaller rank. Indeed, it is verified in Magma that this curve has rank 3, so it is a sharp curve of genus four, and
$$C(\mathbb{Q})=\{\infty_{\pm}, (0,\pm1), (11,\pm 1), (22,\pm1), (33, \pm1)\}.$$

We use a similar strategy to construct a sharp curve of genus five. We choose $p=13$ and start with the reduction curve $\overline{C}:y^2=x^{12}+1$. Then we consider the curve
$$C':y^2=x^4\left(x^2-\dfrac{13^2}{3^2}\right)^2\left(x^2-\dfrac{13^2}{4^2}\right)^2+1,$$
or, after the linear change of variables
$$C:y^2=x^4(9x^2-169)^2(16x^2-169)^2+144^2.$$
This curve is potentially sharp, as $\#\overline{C}(\mathbb{F}_{13})=4$, and $\#C(\mathbb{Q})\geq 12$. Hence, the rank of its Jacobian satisfies $r\geq 4$. \\
The group of automorphisms of $C$ contains the group $(\mathbb{Z}/2\mathbb{Z})^2$, so by \cite[3.1.1., Theorem 5]{DecomposingJacobians}, $C$ has two quotients of degree two, $C_1$, and $C_2$, such that 
$J(C)$ is isogenous to the product $J(C_1)\times J(C_2)$. We conclude that the rank of $J(C)$ is equal to the sum of ranks of $J(C_1)$ and $J(C_2)$. Using Magma, we compute that $C_1$ and $C_2$ are given by the equations\\
$\scriptstyle C_1:y^2 + (x^2 + x)y = 5184x^6 + 37944x^5 - 802992x^4 - 3580910x^3 + 35285157x^2 + 32655315x + 7268209,$\\
$\scriptstyle C_2:y^2 = -20736x^7 - 55584x^6 + 4461023x^5 + 5168390x^4 - 294280801x^3 - 78257128x^2 + 4786378304x + 8098228736.$\\
Furthermore, for both curves $C_1$ and $C_2$, Magma gives the lower rank bound 0, and the upper bound 2. We conclude that both $J(C_1)$ and $J(C_2)$ have rank 2, because the sum of their ranks is at least 4. So the Jacobian of C has rank 4. Hence the curve is sharp and we have
$$C(\mathbb{Q})=\left\{\infty_{\pm},(0,\pm144),\left(\pm\dfrac{13}{3},\pm12\right), \left(\pm\dfrac{13}{4},\pm12\right)\right\}.$$

We note that it might be possible to use the method of Chabuaty and Coleman to find the rational points on $C_2$, and to use that information to find $C(\mathbb{Q})$. However, the curve $C_2$ is not sharp, so it would be more involved to compute the set $C_2(\mathbb{Q})$.

\bibliographystyle{alpha}

\bibliography{References}

\end{document}